\documentclass[12pt]{article}

\usepackage[margin=1in]{geometry} 
\usepackage{amsmath,amsthm,amssymb}

\newcommand{\Z}{\mathbb{Z}}

\newcommand{\F}{\mathbb{F}}

\newcommand{\maxsubgr}[1][n]{\ensuremath{m_{#1}}}
\newcommand{\subgr}[1][n]{\ensuremath{a_{#1}}}
\newcommand{\maxsubmod}[1][n]{\ensuremath{\tilde{m}_{#1}}}

\newcommand{\divides}{\mid}
\newcommand{\submodule}{\leq}

\DeclareMathOperator{\Der}{Der}
\DeclareMathOperator{\Hom}{Hom}
\DeclareMathOperator{\GL}{GL}
\DeclareMathOperator{\mdeg}{mdeg}
\DeclareMathOperator{\BS}{BS}

\newtheorem{theorem}{Theorem}[section]
\newtheorem{corollary}[theorem]{Corollary}
\newtheorem{lemma}[theorem]{Lemma}

\title{Maximal subgroup growth of a few polycyclic groups}
\author{Andrew James Kelley and Elizabeth Ciorsdan Dwyer Wolfe\thanks{This paper was done with the support of the Student Collaborative Research grant at Colorado College.}}

\begin{document}

\maketitle
 
 
 \begin{abstract}
 	We give here the exact maximal subgroup growth of two classes of polycyclic groups. Let 
 	$G_k = \langle x_1, x_2, ..., x_k \mid x_ix_jx_i^{-1}x_j \text{ for all }   i  <  j \rangle$. So 
 	$G_k = \mathbb{Z} \rtimes (\mathbb{Z} \rtimes (\mathbb{Z} \rtimes ... \rtimes \mathbb{Z})$.
 	Then for all $k \geq 2$, we calculate $m_n(G_k)$, the number of maximal subgroups of $G_k$ of index $n$, exactly. 
 	Also, for infinitely many groups $H_k$ of the form
 	$\mathbb{Z}^2 \rtimes G_2$, we calculate $m_n(H_k)$ exactly.
 \end{abstract}

\section{Introduction}
Let $G$ be a finitely generated (f.g.) group. We denote by $a_n(G)$ the number of subgroups of $G$ of index $n$ (which is necessarily finite), and we
denote by $\maxsubgr(G)$ the number of maximal subgroups of $G$ of index $n$. The subgroup growth of $G$ deals with the growth rate
of the function $a_n(G)$ and related functions, such as $\maxsubgr(G)$ or $s_n(G) := \sum_{k = 1}^n a_k(G)$ or of counting 
only the normal subgroups of $G$ of index $n$.

The area of subgroup growth has had some great success. One highlight is the classification of all f.g.\ groups for which
$a_n(G)$ is bounded above by a polynomial in $n$ (see chapter 5 in \cite{Lubotzky-and-Segal}). Also, Jaikin-Zapirain and Pyber
made a great advance in \cite{Jaikin-Zapirain2011},
 where they give a ``semi-structural characterization'' of groups $G$ for which $\maxsubgr(G)$ is bounded above by a polynomial in $n$.

For calculating the \emph{word} growth in a group with polynomial (word) growth, this degree of polynomial growth is given
by nice, simple formula. However, for subgroup growth, it is often very challenging, given
a group $G$ of polynomial subgroup growth, to calculate $\deg(G)$, its degree of polynomial growth:
\[
\deg(G) = \inf \{ \alpha \mid \subgr(G) \leq n^\alpha \text{ for all large $n$} \} = \limsup \frac{\log \subgr(G)}{\log n}.
\]
Similarly for groups $G$ with polynomial maximal subgroup growth, it is often difficult to determine $\mdeg(G)$, where
\[
\mdeg(G) = \inf \{ \alpha \mid \maxsubgr(G) \leq n^\alpha \text{ for all large $n$} \} = \limsup \frac{\log \maxsubgr(G)}{\log n}.
\]
But progress in both areas have been made. In \cite{Shalev_On_the_degree}, Shalev calculated $\deg(G)$ exactly for certain
metabelian groups and for all virtually abelian groups. In \cite{Kelley-some metabelian groups}, the first author 
calculated $\mdeg(G)$ for some metabelian groups, and in \cite{Kelley-dissertation} he does so for all virtually abelian groups.

What is even rarer than calculating $\mdeg(G)$ is to give an exact formula for $\maxsubgr(G)$ (or of $\subgr(G)$).
In \cite{Gelman}, Gelman gives a beautiful, exact formula for $\subgr(\BS(a,b))$, assuming $\gcd(a,b) = 1$, 
where $\BS(a, b)$ is the Baumslag-Solitar group having presentation $\langle x, y \mid y^{-1}x^ay = x^b \rangle$.
Gelman's argument can be easily modified to give an exact formula for $\maxsubgr(BS(a,b))$, where again $\gcd(a,b) = 1$. 
(Alternatively, a different argument, that
explains why $\gcd(a,b) = 1$ is such a nice assumption, is given by the first author in \cite{Kelley-Baumslag-Solitar groups}.)

Since there are so few groups $G$ for which $\maxsubgr(G)$ is known exactly, this paper does so for two infinite classes of
polycyclic groups.

For $k \geq 2$, consider the group $G_k$ with presentation  $\langle x_1, x_2, ..., x_k | x_ix_jx_i^{-1}x_j \text{ for all }   i  <  j \rangle$. Then $G_k$ has the form 
$\Z \rtimes (\Z \rtimes (\Z \rtimes ... \rtimes \Z)$, where the $i$th $\Z$, reading from right to left, is generated by $x_i$. 
Note that the Hirsch length of $G_k$ is $k$, and so if $i \neq j$, then $G_i \not\cong G_j$. In Theorem~\ref{thm:m_n(G_k)},
we calculate $\maxsubgr(G_k)$ exactly for $k \geq 2$.

Let $G_2$ be as above, but write $G_2=\Z \rtimes \Z$ as $\langle b \rangle \rtimes \langle a \rangle$ instead of 
$\langle x_2 \rangle \rtimes \langle x_1 \rangle$. 
For $k \in \Z$, we will define the group $H_k$, which is of the form $\Z^2 \rtimes G_2$. 
The generator $a$ acts (by conjugation) on $\Z^2$ by multiplication by the matrix $A = \left(\begin{smallmatrix} 0 & 1 \\ 1 & 0 \end{smallmatrix}\right)$, 
and the generator $b$ acts (by conjugation) on $\Z^2$ by multiplication by the matrix 
$B_k =  \left(\begin{smallmatrix} 0 & 1 \\ -1 & k \end{smallmatrix}\right)$. Then in Theorem~\ref{thm:exact maximal subgroup growth of H_k},
we calculate $\maxsubgr(H_k)$ exactly for all $k \in \Z$. A consequence of this theorem is that among the groups $H_k$, there
are infinitely many that are pairwise non-isomorphic. Also, it is interesting that $\mdeg(H_2) = 2$, but $\mdeg(H_k) = 1$ for all $k \neq 2$.

For a module $N$, we let $\maxsubmod(N)$ denote the number of maximal submodules of $N$ of index $n$. Also, $M \leq N$ denotes that
$M$ is a submodule of $N$. And $\langle n_1, n_2, ...,n_k \rangle$, where $n_i \in N$ for all $i$, 
denotes the submodule they generate. Recall that when $N$ is a $G$-module, a function $\delta : G \to N$ is called a derivation 
(or a 1-cocycle)
if $\delta(gh) = \delta(g) + g\cdot \delta(h)$ for all $g, h \in G$. The set of derivations from $G$ to $N$ is denoted
$\Der(G, N)$.


\section{Groups of the form  $\Z \rtimes (\Z \rtimes (\Z \rtimes ... \rtimes \Z)$}
Let $G_k$ be as in the introduction. (And let $G_1 = \Z$.) In the following lemma, we will use the fact that if $\delta \in \Der(G, N)$, then
for $g \in G$, we have $\delta(g^{-1}) = -g^{-1}\delta(g)$ which follows from the fact that $\delta(g^{-1}g) = \delta(1) = 0$.


\begin{lemma} \label{lem:General case for generators x_i lemma}
Let $S$ be a simple $G_k$-module. There is a one-to-one correspondence between the set $\Der(G_k, S)$ and the set $\Delta$ of all functions 
	$\delta: \{x_1, x_2,...x_k\} \to S$ satisfying
\begin{equation}
(1-x_j^{-1})\delta(x_i) = (-x_i-x_j^{-1})\delta(x_j) \text{\quad for all $i$, $j$ with $i<j$.} \tag{$*$}
\end{equation}
 
\end{lemma}
 
\begin{proof}
First, let $\delta \in \Der(G_k, S)$. Fix $i$ and $j$ with $i < j$. By the relations of the presentation of $G_k$, we have that $x_ix_j = x_j^{-1}x_i$. Thus, $\delta(x_ix_j) = \delta(x_j^{-1}x_i)$. This gives us \\
\begin{align*}
    \delta(x_ix_j)&=\delta(x_j^{-1}x_i)\\
    \delta(x_i) + x_i\delta(x_j) &= -x_j^{-1}\delta(x_j) + x_j^{-1}\delta(x_i) \\ 
    \delta(x_i) - x_j^{-1}\delta(x_i) &= -x_i\delta(x_j) - x_j^{-1}\delta(x_j) \\ 
    (1-x_j^{-1})\delta(x_i) &= (-x_i - x_j^{-1})\delta(x_j)
\end{align*}
which is the equation in ($*$). 

Next, let $\delta \in \Delta$. By exercise 3(a) in \cite{Brown} (pg.\ 90) (or Lemma 2.20 from \cite{Kelley-dissertation}), 
there is a unique derivation $\delta: F_k \xrightarrow{} S$, where $F_k$ is the
free group on $k$ generators and
the action of $F_k$ on $S$ is the induced action. 

Fix $i$ and $j$ with $i < j$. Then $\delta(x_ix_jx_i^{-1}x_j) = 0$ because
\begin{align*}
    \delta(x_ix_jx_i^{-1}x_j) &= \delta(x_i) + x_i\delta(x_jx_i^{-1}x_j) \\ &= \delta(x_i) + x_i(\delta(x_j) + x_j\delta(x_i^{-1}x_j)) \\ &= \delta(x_i) + x_i(\delta(x_j) + x_j(\delta(x_i^{-1}) + x_i^{-1}\delta(x_j))) \\ &= \delta(x_i) + x_i\delta(x_j) + x_ix_j\delta(x_i^{-1})+x_ix_jx_i^{-1}\delta(x_j) \\ &= \delta(x_i) + x_i\delta(x_j) -x_ix_jx_i^{-1}\delta(x_i)+x_j^{-1}\delta(x_j) \\ &= \delta(x_i) + x_i\delta(x_j) -x_j^{-1}\delta(x_i) + x_j^{-1}\delta(x_j) \\ &= (\delta(x_i) - x_j^{-1}\delta(x_i))- (-x_i\delta(x_j) - x_j^{-1}\delta(x_j)) \\ &= 0.
\end{align*}
Here, the last equality holds because ($*$) holds. Thus, by Lemma 2.19 from \cite{Kelley-dissertation},
which is basically exercise 4(a) in \cite{Brown} (pg.\ 90), we have a derivation $\delta$ from $G_k$ to $S$. 

\end{proof}

\begin{lemma} 
	\label{lem:counting derivations form G_k}
Consider $\Z/p\Z$, a simple $G_k$-module, where each generator $x_i$ of $G_k$ acts on $\Z/p\Z$ by multiplication by $-1$. Then 

\[ \lvert\Der(G_k, \Z/p\Z)\rvert = \begin{cases}
2^k & \text{if } p=2\\
p & \text{if } p \neq 2.

\end{cases}\]
\end{lemma}

\begin{proof} 
 If $p=2$, then the action of $G_k$ on $\Z/2\Z$ is trivial, and so 
$\lvert\Der(G_k, \Z/2\Z)\rvert=\lvert\Hom(G_k,\Z/2\Z)\rvert$, which is $2^k$. 

 Next, suppose $p\neq2$. The action of $(1-x_i^{-1})$ and $(-x_i-x_j^{-1})$ on $\Z/p\Z$ is multiplication by 2, which is invertible since
 $p \neq 2$. So  ($*$) from Lemma~\ref{lem:General case for generators x_i lemma} becomes
 $2\delta(x_i) = 2\delta(x_j)$ for all $i < j$, which simplifies to $\delta(x_i) = \delta(x_j)$ for all $i < j$.
 
So by Lemma \ref{lem:General case for generators x_i lemma}, we may choose a derivation by picking $\delta(x_k)$ to be any element of $\Z/p\Z$, 
and then letting $\delta(x_i) = \delta(x_k)$ for all $i < k$. Thus $\lvert\Der(G_k, \Z/p\Z)\rvert = |\Z/p\Z|=p$.
\end{proof}


\begin{theorem} 
	\label{thm:m_n(G_k)} 
	Let $G_k$ be as above. Then 
\[ \tag{$*$}
\maxsubgr(G_{k}) =  \begin{cases} 1 + (k-1)n & \text{if $n$ is a prime with $n>2$} \\ 
    \sum\limits_{j=0}^{k-1} 2^j & \text{if $n=2$}\\
    0 & \text{if $n$ is not prime}. \end{cases}
\]
\end{theorem}
\begin{proof}
Consider $N$, the subgroup of $G_k$ generated by $x_k$. Then $N \cong \Z$, and 
$N \trianglelefteq G_k$ and $G_k/ N \cong G_{k-1}$. So $N$ is a $G_{k-1}$-module.  So since $G_k \cong N \rtimes G_{k-1}$,
Lemma 5 from \cite{Kelley-some metabelian groups} gives us 
\begin{equation*}
    \maxsubgr(G_k) = \maxsubgr(G_{k-1}) + \sum\limits_{N_0}\lvert \Der(G_{k-1}, N/N_0)\rvert
\end{equation*}
where the sum is over all maximal submodules $N_0$ of $N$ of index $n$. Of course, the maximal submodules of $N$ are precisely the 
subgroups of prime index.
 Thus if $n$ is not prime, then $\maxsubgr(G_k)=0$; this follows by induction on $k$.

Fix a prime $p$. For both cases $p>2$ and $p = 2$, we proceed by induction on $k$.

First, let $p > 2$, and let $k=1$. Then $m_p(G_1) = 1 = 1+(k-1)p$.
Assume ($*$) holds for $k=a$. Then $m_p(G_a)=1+(a-1)p$. Consider $k=a+1$. We have 
\begin{equation*}
   m_p(G_{a+1})= m_p(G_a) + \sum\limits_{N_0} \lvert \Der(G_{a}, N/N_0)\rvert. 
\end{equation*}

By Lemma \ref{lem:counting derivations form G_k}, $\sum\limits_{N_0} \lvert \Der(G_{a}, N/N_0) \rvert = p$. So 
$m_p(G_{a+1}) = 1+(a-1)p + p = 1+ (a+1-1)p$, the desired result.

Finally, let $p=2$, and let $k=1$. Then $m_2(G_1) = 1=2^0$. Assume ($*$) holds for $k=a$. Then $m_2(G_a)= 2^0+2^1+ \cdots + 2^{a-1}$. 
Consider $k=a+1$. Then $\maxsubgr[2](G_{a+1})=\maxsubgr[2](G_a) + \lvert \Der(G_{a}, \Z/2\Z) \rvert$. By 
Lemma \ref{lem:counting derivations form G_k}, $\lvert\Der(G_{a}, \Z/2\Z)\rvert=$ $2^{a}$. Thus $m_{2}(G_{a+1})=2^0+2^1+ \cdots +2^{a-1} + 2^a$, the desired result. 
\end{proof}

 \section{Some groups of the form $\Z^2 \rtimes (\Z \rtimes \Z)$} 
 \label{sec:the groups H_n}

Next, we will define the groups $H_k$, which are of the form $\Z^2 \rtimes (\Z \rtimes \Z)$. We will write $G_2=\Z \rtimes \Z$ as $\langle b \rangle \rtimes \langle a \rangle$ instead of $\langle x_2 \rangle \rtimes \langle x_1 \rangle$. Recall that $G_2=\langle a, b | aba^{-1}b\rangle$.
To form a group of the form $\Z^2 \rtimes (\Z \rtimes \Z)$, what we need is an action
of $G_2$ on $\Z^2$, and so we just need to find matrices $A, B \in \GL_2(\Z)$ such that
$ABA^{-1}B = I_2$, and then we can say that the action (by conjugation) of the generator $a$ on $\Z^2$ is multiplication by the matrix $A$,
and the generator $b$ acts (by conjugation) on $\Z^2$ by multiplication by the matrix $B$.

We will take $A = \begin{pmatrix} 0 & 1 \\ 1 & 0 \end{pmatrix}$. Let $B = \begin{pmatrix} w & x \\ y & z \end{pmatrix}$. Then 
$ABA^{-1}B  = \begin{pmatrix} y^2 + wz & yz + xz \\ wy + xw & wz + x^2 \end{pmatrix}$, and we need this to equal $I_2$.
And so we have $y^2 + wz = 1$, $wz + x^2 = 1$, $wy + xw = 0$ (equivalently, $w = 0$ or $x + y = 0$), and 
$yz + xz = 0$ (equivalently, $z = 0$ or $x + y = 0$). Also, since we need $\det(B) = \pm 1$, we must have $wz - xy = \pm 1$.
One way to solve these equations is to let $w = 0$. Then $x, y = \pm 1$. Take $x = 1$. If we take $y = -1$, then $z$
can be any integer.

Let the group $H_k$ be the group formed when we take $B$ to be $B_k = \begin{pmatrix} 0 & 1 \\ -1 & k \end{pmatrix}$.

\begin{lemma}
	\label{lem:maximal submodules contain pZ^2}
	Let $M$ be a maximal submodule of $\Z^2$. Then $p\Z^2 \submodule M$.
\end{lemma}
\begin{proof}
	First, recall that every maximal subgroup of a polycyclic group has prime power index; this follows, for example, from the proof
	of Result 5.4.3 (iii) in \cite{Robinson}.
	
	Since $M$ yields a maximal subgroup of $H_k$ with index
	equal to $[\Z^2 : M]$, we must have $[\Z^2 : M] = p^j$ for some prime $p$. Therefore, $p^j \Z^2 \submodule M$. Consider the group
	$\Z^2 / p^j \Z^2$. Its Frattini subgroup is $p\Z^2/p^j\Z^2$, and therefore, $p\Z^2/p^j\Z^2 + M/p^j\Z^2 = M/p^j\Z^2$. And hence
	$p\Z^2 + M = M$, that is, that $p\Z^2 \submodule M$.
\end{proof}
For a prime $p$, inside the module $\Z^2$, consider the submodule
 $M_{p, \mathbf{w}} = \left\langle \left(\begin{smallmatrix} p \\ 0 \end{smallmatrix}\right), \left( \begin{smallmatrix} 0 \\ p \end{smallmatrix}\right), \mathbf{w}  \right\rangle$,
 where $\mathbf{w} \in \Z^2$. We will assume $\mathbf{w} \notin p\Z^2$.
 Then $M_{p, \mathbf{w}}$ is a proper (and hence maximal) submodule of $\Z^2$ if and only if $\mathbf{w}$ is an eigenvector
 of both matrices $A$ and $B_k$, considered as elements of $\GL_2(\F_p)$.
 
Let $\mathbf{v} = \left(\begin{smallmatrix} 1 \\ 1 \end{smallmatrix}\right)$ and 
$\mathbf{u} = \left( \begin{smallmatrix} 1 \\ -1 \end{smallmatrix}\right)$, and let
$M_p = M_{p, \mathbf{v}}$ and $M_{p,-1}  = M_{p, \mathbf{u}}$. Of course, $\mathbf{v}$ and $\mathbf{u}$ are eigenvectors of 
$A \in \GL_2(\F_p)$ for all $p$. (And $\mathbf{v} \equiv \mathbf{u}$ (mod 2). Hence, $M_2 = M_{2,-1}$.)

\begin{lemma}
	\label{lem:the maximal submodules of Z^2}
	With the above notation, we have that $M_p$ is a maximal submodule of $\Z^2$ iff $p \divides k - 2$. Also, 
	$M_{p, -1}$ is a maximal submodule of $\Z^2$ iff $p \divides k + 2$. Further, $M_{p, \mathbf{w}}$ is not a proper submodule of $\Z^2$
	unless $M_{p, \mathbf{w}} = M_p$ or $M_{p, -1}$. Thus, $p\Z^2$ is a maximal submodule of $\Z^2$ iff $p \nmid (k - 2)(k + 2)$.
	Finally there are no maximal submodules of $\Z^2$ besides (the appropriate) $M_p$, $M_{p, -1}$, and $p\Z^2$.
\end{lemma}
\begin{proof}
	Let $\mathbf{v}$ and $\mathbf{u}$ be as above, but consider them as elements of $\Z^2/p\Z^2$.
	We have that $B_k \mathbf{v} = \left(\begin{smallmatrix} 1 \\ k - 1 \end{smallmatrix}\right)$, and so
	$B_k \mathbf{v} = \lambda \mathbf{v}$ for some $\lambda \in \Z$ iff $k - 1 \equiv 1$ (mod $p$), i.e.\ iff $p \divides k - 2$.
	
	Also, $B_k \mathbf{u} = \left(\begin{smallmatrix} -1 \\ -1 - k \end{smallmatrix}\right)$, and so $B_k \mathbf{u} = \lambda \mathbf{u}$ for some $\lambda \in \Z$ iff
	$\lambda \equiv -1$ (mod $p$) and $-1 - k \equiv -\lambda$ (mod $p$), and this happens iff $-1 - k \equiv 1$ or $k \equiv -2$ (mod $p$),
	which is equivalent to $p \divides k + 2$.
	
	That no other $M_{p, \mathbf{w}}$ is a proper submodule of $\Z^2$ follows from the fact that any eigenvector of $A$ is a multiple
	of $\mathbf{v}$ or of $\mathbf{u}$. 
	
	Next, let $p \nmid (k - 2)(k + 2)$. Since neither $M_p$ nor $M_{p, -1}$ is a maximal submodule of $\Z^2$ and neither is any other
	$M_{p, \mathbf{w}}$, we have that $p\Z^2$ is a maximal submodule of $\Z^2$. And if $p \divides (k - 2)(k+2)$, then $p \divides k - 2$ or
	$p \divides k + 2$, in which case $M_p$ or $M_{p, -1}$ is a proper submodule of $\Z^2$ that properly contains $p\Z^2$.
	
	The final statement of this lemma follows from the previous parts of this lemma, together with Lemma~\ref{lem:maximal submodules contain pZ^2}:
	Indeed Lemma~\ref{lem:maximal submodules contain pZ^2} implies that either $p\Z^2$ is maximal, or some module containing it is. We are done
	since any proper submodule of $\Z^2$ containing $p\Z^2$ is of the form $M_{p, \mathbf{w}}$.
\end{proof}

For a module $N$, we let $\maxsubmod(N)$ denote the number of maximal submodules of $N$ of index $n$.
\begin{corollary}
	\label{cor:maximal submodule growth of Z^2}
	In what follows, $p$ stands for a prime. We have
	\[
	\maxsubmod(\Z^2) =
	\begin{cases}
	1 & \text{if } n = p \text{ and } p \divides (k - 2)(k + 2) \\
	1 & \text{if } n = p^2 \text{ and } p \nmid (k - 2)(k + 2) \\
	0 & \text{otherwise}.
	\end{cases}
	\]
\end{corollary}
\begin{proof}
	Note that if $p \divides k - 2$ and $p \divides k + 2$, then $k - 2 \equiv k + 2$ (mod $p$), whence $p = 2$. And using the previous notation,
	recall that $M_2 = M_{2,-1}$. 
	
	This corollary then follows from Lemma~\ref{lem:the maximal submodules of Z^2}.
\end{proof}

\begin{lemma}
	\label{lem:derivations from G_2}
Consider $G_2$ with presentation $\langle a, b \mid aba^{-1}b\rangle$, as described above. Let $S$ be a simple $G_2$-module. 
Then there is a one-to-one correspondence between the set
 $\Der(G_2, S)$ and the set of functions $\delta: \{a,b\}\xrightarrow{} S$ satisfying 
\begin{equation*}
   (1-b^{-1}) \delta(a) = (-a-b^{-1})\delta(b). \tag{$*$}
\end{equation*}
\end{lemma}
\begin{proof}
 This follows from Lemma $\ref{lem:General case for generators x_i lemma}$.
\end{proof}



\begin{lemma}
	\label{lem:counting derivations from G_2}
	Let $S$ be a simple $G_2$-module with $S$ either $\Z^2/M_p$ or $\Z^2/M_{p, -1}$ or $\Z^2/p\Z^2$. Then
	\[
	\lvert \Der(G_2, S) \rvert = 
	\begin{cases}
	|S|^2 = p^2 & \text{if } S = \Z^2/M_p \\
	|S| = p & \text{if } S = \Z^2/M_{p, -1} \text{ and } p > 2 \\
	|S| = p^2 & \text{if } S = \Z^2/p\Z^2.
	\end{cases}
	\]
\end{lemma}
\begin{proof}
	The element $1 - b^{-1}$ in ($*$) from Lemma~\ref{lem:derivations from G_2} acts on $\delta(a)$ by multiplication by the matrix
	$I_2 - B_k^{-1} = \left( \begin{smallmatrix} 1 - k & 1 \\ -1 & 1 \end{smallmatrix} \right)$ which has determinant $2 - k$, and the element
	$-a - b^{-1}$ acts by multiplication by the matrix $-A - B_k^{-1} = \left( \begin{smallmatrix} -k & 0 \\ -2 & 0 \end{smallmatrix} \right)$.
	
	First, suppose $S \neq \Z^2/M_p$. Then $S = \Z^2/M_{p, -1}$ with $p > 2$ (since $M_2 = M_{2, -1}$) or
	$S = \Z^2/p\Z^2$. Either way, with $p$ thus determined by $S$, we have that $M_p$ is not a maximal submodule of $\Z^2$
	since either $p\Z^2$ is maximal or $M_{p, -1}$ is and $\maxsubmod[p](\Z^2) \leq 1$. 
	Hence $p \nmid k - 2$ by Lemma~\ref{lem:the maximal submodules of Z^2}. 
	In this case, $I_2 - B_k^{-1}$ is invertible, considered as an element of $\GL_2(\F_p)$. Hence ($*$) from
	Lemma~\ref{lem:derivations from G_2} may be written as 
	\[
	\delta(a) = (I_2 - B_k^{-1})^{-1}(-A - B_k^{-1}) \delta(b).
	\]
	And so in this case, we are free to choose $\delta(b)$ to be any element of $S$, and then this determines what $\delta(a)$ must be.
	Thus we would have $\lvert \Der(G_2, S) \rvert = |S|$.
	
	Next, suppose $S = \Z^2/M_p$. Then the maximality of $M_p$ implies (by Lemma~\ref{lem:the maximal submodules of Z^2})
	that $p \divides k - 2$. And so then $I_2 - B_k^{-1} \equiv \left( \begin{smallmatrix}-1 & 1 \\ -1 & 1 \end{smallmatrix} \right)$ and
	$-A - B_k^{-1} \equiv \left( \begin{smallmatrix} -2 & 0 \\ -2 & 0 \end{smallmatrix} \right)$ (mod $p$).
	
	We have that $\left\{\left( \begin{smallmatrix} i \\ 0 \end{smallmatrix} \right) : 0 \leq i < p \right\}$ is a complete set of representatives of 
	$\Z^2/M_p$. Then letting $\delta(a) = \left( \begin{smallmatrix} i \\ 0 \end{smallmatrix} \right) + M_p$ and 
	$\delta(b) = \left( \begin{smallmatrix} j \\ 0 \end{smallmatrix} \right) + M_p$, we have that equation ($*$) from Lemma~\ref{lem:derivations from G_2}
	holds because $(I_2 - B_k^{-1}) \left( \begin{smallmatrix} i \\ 0 \end{smallmatrix} \right)  = \left( \begin{smallmatrix} -i \\ -i \end{smallmatrix} \right) \in M_p$ and
	$(-A - B_k^{-1}) \left( \begin{smallmatrix} j \\ 0 \end{smallmatrix} \right) = \left( \begin{smallmatrix} -2j \\-2j \end{smallmatrix} \right) \in M_p$. And so both
	$(1-b^{-1}) \delta(a)$ and $(-a-b^{-1})\delta(b)$ are the trivial element of $\Z^2/M_p$. Therefore, in this case, we have
	$|S|^2$ options for a derivation from $G_2$ to $S$.
		
\end{proof}

  \begin{theorem}
  	\label{thm:exact maximal subgroup growth of H_k} 
  	In what follows, $p$ stands for a prime. We have
  	\[
  	\maxsubgr(H_k) = 
  	\begin{cases}
  	n^2 + n + 1 & \text{if } n = p \text{ and } p \divides k- 2\\
  	2n + 1  & \text{if } n = p \text{ and } p \divides k + 2 \text{ with } p > 2\\
  	n + 1 & \text{if } n = p \text{ and } p \nmid (k-2)(k+2)\\
  	n & \text{if } n = p^2 \text{ and } p \nmid (k -2)(k+2) \\
  	0 & \text{otherwise}.
  	\end{cases}
  	\]
  \end{theorem} 
  \begin{proof}
  Consider $\Z^2 \trianglelefteq \Z^2 \rtimes G_2$. By Lemma 5 from \cite{Kelley-some metabelian groups}, we have 
  \[\tag{1}
      \maxsubgr(H_k) = \maxsubgr(G_2) + \sum_{N_0} \lvert \Der(G_2, \Z^2 / N_0)\rvert
  \]
  where the sum is over all maximal submodules $N_0$ of $\Z^2$ of index $n$. 
  Also, by Theorem~\ref{thm:m_n(G_k)}, we have 
  \[\tag{2}
      \maxsubgr(G_2)= 
      \begin{cases} 
     n+1 & \text{if } n \text{ is a prime}\\
	 0 & \text{otherwise}.
 \end{cases} 
  \]
  We have that Lemmas~\ref{lem:the maximal submodules of Z^2} and \ref{lem:counting derivations from G_2}
  (and Corollary~\ref{cor:maximal submodule growth of Z^2})
   together imply that
  \[\tag{3}
  \sum_{N_0} \lvert \Der(G_2, \Z^2 / N_0)\rvert = 
  \begin{cases}
  n^2 & \text{if } n = p \text{ and } p \divides k - 2\\
  n & \text{if } n = p \text{ and } p \divides k +2 \text{ with } p > 2\\
  n & \text{if } n = p^2 \text{ and } p \nmid (k - 2)(k + 2)\\
  0 & \text{otherwise}.
  \end{cases}
  \]
  The present theorem follows from (1) by adding (2) and (3).
  \end{proof}

For $n \in \Z$, let $\pi(n)$ denote the set of prime numbers dividing $n$. Then a consequence of Theorem~\ref{thm:exact maximal subgroup growth of H_k}
is that for $i, j \in \Z$, if $\pi(i - 2) \neq \pi(j - 2)$ or $\pi(i + 2) \neq \pi(j+ 2)$, then $H_i \not\cong H_j$.
Also, note that $\mdeg(H_2) = 2$ (because $\pi(0)$ is the set of all primes), and for all $k \neq 2$, $\mdeg(H_k) = 1$.

\textsc{Department of Mathematics and Computer Science, Colorado College,
	Colorado Springs, Colorado 80903}\par\nopagebreak
\textit{email address}: \texttt{akelley@coloradocollege.edu}\footnote{The first author is a Visiting Assistant Professor of Mathematics at Colorado College.
	A possibly more permanent email address is \texttt{akelley2500@gmail.com} }

\vspace{.1in}

\textsc{Department of Mathematics and Computer Science, Colorado College,
	Colorado Springs, Colorado 80903}\par\nopagebreak
\textit{email address}: \texttt{e\_wolfe@coloradocollege.edu}

\end{document}